\documentclass[11pt]{article}

\usepackage[a4paper,vmargin={3.5cm,3.5cm},hmargin={2.5cm,2.5cm}]{geometry}
\usepackage[font=sf, labelfont={sf,bf}, margin=1cm]{caption}
\usepackage{dsfont,amsmath,amsfonts,amsthm,amssymb}

\usepackage[toc,page]{appendix} 
\usepackage{graphicx}
\usepackage{graphics}
\usepackage{stmaryrd}
\usepackage{verbatim}
\usepackage{hyperref}
\usepackage{color}
\usepackage{bbm}
\usepackage[T1]{fontenc}
\usepackage[cyr]{aeguill}
\date{}

\title{\bf{\textsc{Local explosion in self-similar \\growth-fragmentation processes}}}

\author{Jean Bertoin\thanks{ Institut f\"ur Mathematik, 
Universit\"at Z\"urich, 
Winterthurerstrasse 190, 
CH-8057 Z\"urich, Switzerland. \hfill \eject
Email: jean.bertoin@math.uzh.ch}  \quad \& \hspace{0.2cm}Robin Stephenson\thanks{Institut f\"ur Mathematik, 
Universit\"at Z\"urich, 
Winterthurerstrasse 190, 
CH-8057 Z\"urich, Switzerland. \hfill \eject E-mail: robin.stephenson@normalesup.org}}

\begin{document}

\maketitle

\newtheorem{theo}{Theorem}[section]
\newtheorem{lemma}{Lemma}[section]
\newtheorem{prop}{Proposition}[section]
\newtheorem{cor}{Corollary}[section]
\newtheorem{defi}{Definition}[section]
\newtheorem{rem}{Remark}
\numberwithin{equation}{section}
\newcommand{\e}{{\mathrm e}}
\newcommand{\rep}{{\mathrm{rep}}}
\newcommand{\R}{{\mathbb{R}}}
\newcommand{\C}{{\mathbb{C}}}
\newcommand{\T}{\mathbf{T}}
\newcommand{\E}{\mathbb{E}}
\renewcommand{\L}{\mathbf{L}}
\newcommand{\TT}{\mathbb{T}}
\newcommand{\s}{\mathcal{S}^{\downarrow}}
\newcommand{\p}{\mathcal{P}}
\newcommand{\N}{\mathbb{N}}
\newcommand{\D}{\mathcal{D}}
\newcommand{\W}{\mathcal{W}}
\newcommand{\pr}{\mathbb{P}}
\newcommand{\Q}{\mathbb{Q}}
\newcommand{\z}{\mathbf{\zeta}}
\newcommand{\m}{\mathbf{\mu}}
\newcommand{\M}{\mathcal{M}}
\newcommand{\Z}{\mathbb{Z}}
\newcommand{\dia}{\diamond}
\newcommand{\bul}{\bullet}
\renewcommand{\geq}{\geqslant}
\renewcommand{\leq}{\leqslant}
\newcommand{\veps}{\varepsilon}

\abstract{Markovian growth-fragmentation processes describe a family of particles which can grow larger or smaller with time, and occasionally split in a conservative manner. They were introduced in \cite{Bertoingrowthfrag}, where special attention was given to the \emph{self-similar} case. A Malthusian condition was notably given under which the process does not locally explode, in the sense that for all times, the masses of all the particles can be listed in non-increasing order. Our main result in this work states the converse: when this condition is not verified, then the growth-fragmentation process explodes almost surely. Our proof involves using the additive martingale to bias the probability measure and obtain a spine decomposition of the process, as well as properties of self-similar Markov processes.}

\medskip

\noindent \emph{\textbf{Keywords:}  Growth-fragmentation, self-similarity, branching process, spine decomposition}

\medskip

\noindent \emph{\textbf{AMS subject classifications:}} 60F17, 60G51, 60J25, 60J80
\section{Introduction}
Informally, a growth-fragmentation process can be viewed as a branching particle system, in which each particle has a mass that evolves continuously (and in particular may grow)  as time passes, independently of the other particles, and then splits in two. When a split occurs, mass is conserved in the sense that the sum of the masses of the two new particles is equal to that of the particle that just split. We may think for instance of growth-fragmentations as a model for cell division, see e.g. \cite{DHKR}. The process is further called self-similar when it fulfills a scaling property. 

More precisely, a self-similar growth-fragmentation
can be defined as follows, as was done in \cite{Bertoingrowthfrag}: start with a positive self-similar Markov process $X=(X(t),t\geq0)$ with no positive jumps, and such  that $X$ either  is absorbed at $0$ after a finite (random) time or converges to $0$  in infinite time. 
We look at $X(t)$ as the mass of a particle at time $t$ and whenever $X$ makes a (necessarily negative) jump, we consider this as giving birth to a new particle, whose original mass is equal to the size of the jump. The new particle then grows and splits just as the original one, in turn begetting new particles, and so on. Note that when the set of times at which a particle jumps (i.e. reproduces) is assumed to be  discrete, this description fits the framework \cite{jagers} of Crump-Mode-Jagers branching processes.

One of the main results established in \cite{Bertoingrowthfrag} is that there is a simple Malthusian condition, which is given in terms of the characteristics of the self-similar Markov process $X$, that ensures that a.s., for all times $t\geq 0$, the particles generated by the growth-fragmentation can be listed in the non-increasing order of their masses and then form a null sequence. In short, our purpose in the present work is to show that conversely, when this condition fails, the self-similar growth-fragmentation explodes, in the sense that for every non-empty open interval $I$ in $(0,\infty)$, there is a random time at which $I$ contains infinitely many particles. 

There is already a vast literature dealing with explosion for various types of branching processes; see in particular the recent works by Amini {\it et al.} \cite{ADGO} for age-dependent branching processes, by Komj\'athy  \cite{Komjathy} for Crump-Mode-Jagers processes, and the references listed therein. In this regard, growth-fragmentations should be considered as spatial branching processes, viewing the evolution of the mass of a particle  as a spatial displacement in the positive half-line. Plainly, the total number of particles may become infinite (which would be often referred to as an explosion in the literature) without inducing the \emph{local} explosion phenomenon in which we are interested here. Typically, this is the case when the system produces in finite time particles with arbitrarily small masses, but only finitely many particles with masses at least $\veps$ for every $\veps>0$.

The fact that self-similar growth-fragmentations may explode has been first pointed out in Section 3 of \cite{BW15} for a very specific set of parameters. 
Roughly speaking, the idea in  \cite{BW15} is that there is a natural genealogical line along which the mass of the particle reaches $0$ continuously in finite time; 
the new particles which branch off this specified line start very close to zero, and one just needs show that enough of them reach the target interval at approximately the same time.
Here, we shall follow the same general idea, but the generality of our statements means that additional work will be needed.

Our main theorem will be stated at the end of the next section, after some preliminaries about the construction of self-similar growth-fragmentations and the introduction of  relevant notation.
The three main ingredients for its proof are then developed in Section 3. As a first step, we analyze a truncation procedure and show that, without loss of generality, we may assume  that the intensity of birth events  is finite. Next, we dwell on changes of probability measures based on certain additive martingales and on the so-called spinal decomposition for homogeneous growth-fragmentations. The last ingredient consists in proving that under an appropriate hypothesis, self-similar growth-fragmentations may in some sense start from $0^+$ if the index of self-similarity $\alpha$ is strictly negative, and from $+\infty$ if $\alpha>0$. This is essentially a consequence of the fact that a similar property holds for certain positive self-similar Markov processes.
The proof of our main theorem will then be completed in the final section.

\section{Preliminaries, notation, and main result}

In order to make this model adapted to the upcoming proofs as well as better fitted with the theory of \cite{Bertoincompfrag}, we will consider a slightly more general version than that described informally in the introduction, where some jumps do not give birth to a new particle. To be precise, recall first that $X$ is characterized by an index of self-similarity $\alpha\in\R$ and the Laplace exponent $\Psi$ of a spectrally negative L\'evy process  given by
\[\Psi(q)=-k+\frac{1}{2} \sigma^2 q^2 + bq + \int_{(-\infty,0)}\left(\e^{qy}-1+q(1-\e^y)\right) \Lambda(\mathrm dy)\,,\qquad q\geq 0.\]
Here, $k\geq 0$ is the killing rate, $\sigma^2\geq 0$ the Gaussian coefficient, $b\in\R$ the drift coefficient, and $\Lambda$ the L\'evy measure which governs the rate of the jumps and fulfills $\int_{(-\infty,0)}(1\wedge y^2) \Lambda(\mathrm dy)<\infty$. $X$ can then be constructed as the Lamperti transform of aforementioned L\'evy process (see below for details). We implicitly assume that the killing rate $k$ is not $0$, or that the right-derivative of $\Psi$ satisfies $\dot \Psi(0^+)<0$, which is the necessary and sufficient condition for $X$ to be absorbed at $0$ after a finite time or to converges to $0$ in infinite time. 
We now suppose that $\Lambda$ is given as the sum of two measures,  $\Lambda=\Lambda_1+\Lambda_2$, so that the jumps of $X$ can be seen as coming from two independent Poisson point processes, one corresponding to $\Lambda_1$ and the other to $\Lambda_2$. We will say that a new particle arises from a jump of $X$ if this jump comes from the measure $\Lambda_1$, but not from $\Lambda_2$. 

In the next section, we give a precise construction of the growth-fragmentation process, first assuming that the total mass of $\Lambda_1$ is finite and then treating the general case. This construction is reminiscent of the ``branching L\'evy process" from \cite{Bertoincompfrag}. We then introduce some important notions and notation, and finally state our main theorem.

\subsection{A construction of growth-fragmentations  by truncation}\label{subsec:construction}

Let $\mathcal{U}=\cup_{n=0}^{\infty} \{0,1\}^n$ be the infinite binary tree. Elements of $\mathcal{U}$ are written as $\mathbf{u}=u_1u_2\ldots u_n$ where $u_i\in\{0,1\}$ and $n=|\mathbf{u}|\geq0$ is the generation of $\mathbf{u}$. As usual, the ancestor, that is the unique element at generation $0$, is the empty word $\emptyset$.  If $n>0$ then we let $\mathbf{u}^{-}=u_1\ldots u_{n-1}$ be the parent of $\mathbf{u}$.
We will build the growth-fragmentation  as the multiset (i.e. elements may be repeated and appear with their multiplicity) valued process 
$$\mathbf{X}(t)=\{\!\!\{X_{\mathbf{u}}(t):\;\mathbf{u}\in\mathcal{U}\text{ and }b_{\mathbf{u}}\leq t<d_{\mathbf{u}}\}\!\!\}, \qquad t\geq 0,$$
 where $X_{\mathbf{u}}$ is the size of the particle $\mathbf{u}$ at time $t$, and $b_{\mathbf{u}}$ and $d_{\mathbf{u}}$ are its birth and death times. 
 
 We start with the {\it homogeneous} case when the self-similarity index $\alpha$ is equal to $0$, and further assume that $\Lambda_1$ is a finite measure. The idea is that each particle, say $\mathbf{u}$, evolves as the exponential of a (possibly killed) L\'evy process up to an independent random time which has the exponential law with parameter $\Lambda_1((-\infty,0])$. If the particle is still alive at that time, then it splits. That is, the particle $\mathbf{u}$ then dies giving birth to two children. The left child $\mathbf{u}0$ gets a fraction $\e^{J}$ of the mass, where the distribution of $J$ is $\Lambda_1$ renormalized, and the right child $\mathbf{u}1$ has the complementary mass.

Specifically, for all $\mathbf{u}\in\mathcal{U}$, let $\xi_{\mathbf{u}}=(\xi_{\mathbf{u}}(t),0\leq t < \zeta_{\mathbf{u}})$ be a L\'evy process with Laplace exponent $\Psi_2$ defined by
\[\Psi_2(q)=-k+\frac{1}{2}\sigma^2q^2+\Big(b+\int_{(-\infty,0)}(1-\e^y)\Lambda_1(\mathrm dy)\Big) q+\int_{(-\infty,0)} (\e^{qy}-1+q(1-\e^y))\Lambda_2(\mathrm dy).\] (Note that the drift term has changed due to the compensation term for $\Lambda_1$ which is otherwise not taken into account.) The lifetime $\zeta_{\mathbf{u}}$ of $\xi_{\mathbf{u}}$ follows the exponential distribution with parameter $k$, and in particular $\zeta_{\mathbf{u}}=\infty$ a.s. if $k=0$. 
Let as well $T_{\mathbf{u}}$ be an exponential random variable with parameter $\Lambda_1((-\infty,0])$ and $J_{\mathbf{u}}$ be a random variable with distribution $\frac{1}{\Lambda_1((-\infty,0])}\Lambda_1.$ We take all of these independent, and write $\pr$ for the law of the family of triples $\left(\xi_{\mathbf u}, T_{\mathbf u}, J_{\mathbf u}\right )_{{\mathbf u}
\in{\mathcal U}}$. 

Now we can build a homogeneous growth-fragmentation (which is also called a compensated fragmentation in \cite{Bertoincompfrag})  recursively on the generations, using the notation $\chi, \beta, \delta$ rather than $X,b,d$ for the sake of avoiding later on a possible confusion with the self-similar case $\alpha \neq 0$.
We first  let $\chi_{\emptyset}(t)=\exp(\xi_{\emptyset}(t))$ for $0=\beta_{\emptyset}\leq t < \delta_{\emptyset}= \zeta_{\emptyset}\wedge T_{\emptyset}$. 
Next  for $\mathbf{u}\neq \emptyset$, we write ${\mathbf v}={\mathbf u}^-$ for the parent of ${\mathbf u}$ and assume first that
 $\chi_{\mathbf{v}}(\delta_{\mathbf{v}}^-)>0$. That is $T_{\mathbf v}< \zeta_{\mathbf v}$ and the particle ${\mathbf v}$ is still alive at age $T_{\mathbf v}$. We then  let $\beta_{\mathbf{u}}=\delta_{\mathbf{v}}$ and $\chi_{\mathbf{u}}(\beta_{\mathbf{u}})=\chi_{\mathbf{v}}(\delta_{\mathbf{v}}^-)\e^{J_{\mathbf{v}}}$ if ${\mathbf{u}}={\mathbf{v}}0$ and $\chi_{\mathbf{u}}(\beta_{\mathbf{u}})=\chi_{\mathbf{v}}(\delta_{\mathbf{u}}^-)(1-\e^{J_{\mathbf{v}}})$ if ${\mathbf{u}}={\mathbf{v}}1$. Further we let $\delta_{\mathbf{u}}=\beta_{\mathbf{u}}+\zeta_{\mathbf{u}}\wedge T_{\mathbf{u}}$ and finally, for $t\in[\beta_{\mathbf{u}},\delta_{\mathbf{u}})$, we let $\chi_{\mathbf{u}}(t)=\chi_{\mathbf{u}}(\beta_{\mathbf{u}})\e^{\xi_{\mathbf{u}}(t-\beta_{\mathbf{u}})}$. On the other hand,  if $\chi_{\mathbf{v}}(\delta_{\mathbf{v}}^-)=0$, that is if the particle ${\mathbf v}$ already has died before reaching the age $T_{\mathbf v}$,
 then for definitiveness we let $\beta_{\mathbf u}= \delta_{\mathbf u}=\infty$ and agree that $\chi_{\mathbf u}(\delta_{\mathbf u}^-)= 0$.

For the general {\it self-similar} case $\alpha\neq 0$, the growth-fragmentation process is obtained by applying a standard Lamperti-type  time-change (see \cite{Lam72}) to the homogenous process $\boldsymbol{\chi}$  constructed in the above paragraph. We also introduce the following notation: if ${\mathbf{u}}\in\mathcal{U}$ and $t<\beta_{\mathbf{u}}$, we let $\bar \chi_{\mathbf{u}}(t)$ be equal to $\chi_{\mathbf{v}}(t)$, where $\mathbf{v}$ is the only ancestor of ${\mathbf{u}}$ to be alive at time $t$ in $\boldsymbol{\chi}$ . We also let
$$b_{\mathbf{u}} = \int_0^{\beta_{\mathbf{u}}} \bar \chi_{\mathbf{u}}(r)^{-\alpha}\mathrm{d}r, \quad 
d_{\mathbf{u}} = \int_0^{\delta_{\mathbf{u}}} \bar \chi_{\mathbf{u}}(r)^{-\alpha}\mathrm{d}r,$$
and 
$\tau_{\mathbf{u}}$ be the time-change defined by
	\[\tau_{\mathbf{u}}(t) = \inf \Big\{s\geq 0: \int_0^s \bar \chi_{\mathbf{u}}(r)^{-\alpha}\mathrm{d}r >t\Big\}.
\]
Finally, we let $X_{\mathbf{u}}(t)=\chi_{\mathbf{u}}(\tau_{\mathbf{u}}(t))$ for $b_{\mathbf{u}} \leq t < d_{\mathbf{u}}$. 

\begin{rem}\label{rem:eve} Consider for $t\geq0$ the left-most word $\mathbf{u}(t)=00\ldots 0$, where the number of zeroes is such that $b_{\mathbf{u}(t)}\leq t<d_{\mathbf{u}(t)}.$  Then the line of descent $\mathbf{u}(t)$ can be seen as the Eve of the growth-fragmentation process with the notations of \cite{Bertoingrowthfrag}, and the process $(X_{\mathbf{u}(t)}(t),t\geq0)$ is a positive self-similar Markov process with characteristics $(\Psi,\alpha)$, meaning that it is constructed from a spectrally negative L\'evy process with Laplace exponent $\Psi$ and the Lamperti transformation with parameter $\alpha$ that has been described above. 
\end{rem}
\begin{rem} We will sometime need to have a version of the process which starts with a particle of size $x\neq 1$. In this case, the construction is the same, except that we start with $X_{\emptyset}(0)=x$, that is $\xi_{\emptyset}(0)=\ln x$.  We call $\pr_x$ this distribution.
\end{rem}

When the measure $\Lambda_1$ has infinite total mass, the above construction is not possible since there would be infinitely many branching events in a bounded time interval. Instead, we are going to use an approximation scheme to define $\mathbf{X}$. For $\veps>0$, we let 
$$\Lambda_1^{(\veps)}(\mathrm dx)=\mathbbm{1}_{\{x<-\veps\}}\Lambda_1(\mathrm dx), \hbox{ and }\Lambda_2^{(\veps)}(dx)=\Lambda_2(\mathrm dx)+\mathbbm{1}_{\{x\geq -\veps\}}\Lambda_1(\mathrm dx).$$
 The effect of this is that, when a particle splits in two, if the second child is too small (having a fraction smaller than $1-\e^{-\veps}$ of its parent's mass), then we ``erase" it, which we signify by shifting the corresponding part of $\Lambda_1$ into $\Lambda_2$. 

Plainly, $\Lambda_1^{(\veps)}$ is a finite measure, and we then let $\mathbf{X}^{(\veps)}$ be the self-similar growth-fragmentation processes with corresponding L\'evy measures $\Lambda_1^{(\veps)}$ and $\Lambda_2^{(\veps)}$. By standard properties of L\'evy processes (see Lemma 3 and Equation (19) in \cite{Bertoincompfrag}), the $\mathbf{X}^{(\veps)}$ can naturally be coupled such that the multisets
\[\mathbf{X}^{(\veps)}(t)=\{\!\!\{X^{(\veps)}_{\mathbf{u}}(t):\;\mathbf{u}\in\mathcal{U}\text{ and }b^{(\veps)}_{\mathbf{u}}\leq t<d^{(\veps)}_{\mathbf{u}}\}\!\!\}\]
are increasing as $\veps$ decreases. And thus we can define the general growth-fragmentation process $\mathbf{X}(t)$ as the increasing limit of $\mathbf{X}^{(\veps)}(t)$ as $\veps \to 0^+$.

\subsection{Cumulant and additive martingales}
For $q\geq 0$, we let \[
\kappa(q)=\Psi(q)+\int_{(-\infty,0)} (1-\e^y)^q\Lambda_1(\mathrm dy).\]

The function $\kappa$: $\R_+\to (-\infty,\infty]$ is convex, and takes finite values on $[2,\infty)$ at least. 
Note also that $\kappa(q)$ is finite for all $q\geq 0$ if and only if the measure $\Lambda_1$ is finite. The function $\kappa$ acts as cumulant generating function for the mean intensity of the \emph{homogeneous} growth-fragmentation process $\boldsymbol{\chi}$,  in the sense that we have, for all $t\geq0$ and $q>0$ such that $\kappa(q)<\infty$,
\begin{equation}\label{Emmo}
\E\Big[\sum_{{\mathbf u}\in{\mathcal U}} \chi^q_{\mathbf u}(t)\Big]=\e^{t\kappa(q)}.
\end{equation}
(We stress that in the sum above, it is  implicitly agreed that only the ${\mathbf u}\in{\mathcal U}$ with $\beta_{\mathbf u}\leq t < \delta_{\mathbf u}$ are taken into account.)
As a consequence, for $q$ such that $\kappa(q)<\infty$, the process $M_q=(M_q(t),t\geq0)$ defined by
\[M_q(t)=\e^{-t\kappa(q)}\sum_{{\mathbf u}\in{\mathcal U}} \chi^q_{\mathbf u}(t)\]
is a martingale, which we call an \emph{additive martingale}. This was introduced in \cite{Bertoincompfrag}; see in particular Theorem 1 and Corollary 3 there.

\subsection{Main result}
It is known from \cite{Bertoingrowthfrag} that, if $\alpha=0$ \emph{or} if $\kappa$ takes a non-positive value, then, almost surely,  for all $t\geq 0$ and $\veps>0$,  there are only finitely many particles in $(\veps, \infty)$,  or equivalently,  the population of $\mathbf{X}(t)$ can be ranked in non-increasing order. We will now establish the converse under a slight  technical condition on $\kappa$ that will be enforced throughout the rest of this work:
$$ \hbox{either $\kappa(0^+)=\infty$ or $\dot \kappa(0^+)<0$,}$$
where $\dot \kappa$ stands for the right-derivative of the convex function $\kappa$.
We stress that this assumption is very mild. For instance, it is always fulfilled if $\dot \Psi(0^+)<0$ since obviously $\dot \kappa\leq \dot \Psi$ (in particular, it is always fulfilled if the killing rate $k=0$), or if $\Lambda_1$ has infinite total mass (since then $\kappa(0^+)=\infty$).

If $\alpha=0$, then the number of particles of the homogeneous growth-fragmentation process $\boldsymbol{\chi}$  is simply a  branching process. We say that there is \emph{extinction} if this process dies out, that is if there exists $t>0$ such that there are no particles alive at time $t$: $\{{\mathbf u}\in\mathcal{U}: \beta_{\mathbf u}\leq t < \delta_{\mathbf u}\}=\emptyset$. 
Implicitly ruling out the degenerate case $\Lambda_1=0$ when no birth event ever happen, we see by letting $q\to 0^+$ in \eqref{Emmo} that extinction occurs a.s. if and only if $\kappa(0)\leq 0$. 

If $\alpha\neq 0$, assuming that $\mathbf{X}$ is constructed from a homogeneous version $\boldsymbol{\chi}$  as before, then we say that $\mathbf{X}$ \emph{dies suddenly} if $\boldsymbol{\chi}$  goes extinct. Note that, because of the Lamperti time-change, it is possible for $\mathbf{X}$ to reach $\emptyset$ continuously in the sense that $\boldsymbol{\chi}$  does not go extinct and nonetheless $\mathbf{X}(t)=\emptyset$ for $t$ sufficiently large. For example, Corollary 3 in \cite{Bertoingrowthfrag} shows that this happens   if $\alpha<0$ and $\inf \kappa< 0$.

We  introduce the following fundamental assumption:
\begin{equation}\tag{$\mathbf{H}$}\label{hyp}
\forall q\geq 0, \kappa(q)>0,
\end{equation}
and claim that under $(\mathbf{H})$ and for $\alpha\neq 0$, the growth-fragmentation process explodes almost surely in finite time in the following sense. 

\begin{theo}\label{thprincipal} We assume $(\mathbf{H})$ and $\alpha \neq 0$, and 
restrict ourselves to the event where $\mathbf{X}$ does not suddenly die.  
Then, for every $0<a<a'$, there exists a.s. a random time $t>0$ such that $\mathbf{X}(t)$ has infinitely many elements in the open interval $(a,a')$.
\end{theo}

\begin{rem}\label{rem:extinction} Under $(\mathbf{H})$, we have $\kappa(0^+)>0$ and therefore the probability that $\mathbf{X}$ does not suddenly die is strictly positive
(note that $\kappa(0^+)>0$ also rules out the case when $\Lambda_1=0$). \end{rem}

 Actually, a simple generalization of our argument yields a stronger version of Theorem \ref{thprincipal}. Specifically, for every $0< a_1< a'_1<a_2<a'_2<\cdots<a_n<a'_n$, there exists a.s. a random time $t>0$ such that $\mathbf{X}(t)$ has infinitely many elements in each of the intervals $(a_i,a_i')$. However, for the sake of simplicity, we shall concentrate on the case of a single interval.

\section{Truncation, tilting, and starting from a boundary}

Throughout this section, it is assumed that $(\mathbf{H})$ holds.
\subsection{Reduction by truncation}

We start with a simple general statement about the minimum of $\kappa.$
\begin{lemma} \label{L1} The cumulant function $\kappa$ reaches its positive minimum on $[0,\infty)$ at a value $q_m>0$.
\end{lemma}

\begin{proof} We investigate the limits of $\kappa$ when $q$ tends to the boundaries of its domain $\{q\geq 0: \kappa(q)<+\infty\}$. First for $q\to +\infty$, notice that the integral term in $\kappa-\Psi$ converges to $0$ by dominated to convergence, implying that $\kappa$ has the same limit as $\Psi$ at infinity. Taking the limit as $q\to \infty$ in the formula $\e^{\Psi(q)}=\E[\e^{q\xi(1)}]$ where $\xi$ is a L\'evy process with Laplace exponent $\Psi$, it is clear that $\Psi(+\infty)=+\infty$ if $\pr(\xi(1)>0)>0$, and $\Psi(+\infty)<0$ otherwise. Thus, assuming $(\mathbf{H})$, we have $\kappa(+\infty)=+\infty$.

Recall then that $\kappa $ is convex and introduce
$$\underline{q}=\inf\{q\geq 0: \kappa(q)<\infty\}=\inf\left \{q\geq 0: \int_{(-\infty,0)} (1-\e^y)^q\Lambda_1({\mathrm d}y)<\infty\right \}.
$$
Note that, by monotone convergence, $\kappa(q)$ converges to $\kappa(\underline{q})$ as $q$ decreases to $\underline{q}$. If $\kappa(\underline{q})=+\infty$ then $\kappa$ reaches its minimum on the open interval $(\underline{q},\infty)$, while if $\kappa(\underline{q})<\infty$, then it is reached on $[\underline{q},\infty)$. This proves our claim if either  $\underline{q}=0$ and $\kappa(0)=+\infty$, or $\underline{q}>0$. It remains to consider the case when $\kappa(0)<+\infty$. But then we have assumed that $\dot \kappa(0+)<0$, and the same conclusion follows. 
\end{proof}

Next for $\veps>0$, we introduce 
\[\kappa^{(\veps)}(q)=\Psi(q)+\int_{(-\infty,-\veps)} (1-\e^y)^q\Lambda_1(\mathrm dy),\]
that is $\kappa^{(\veps)}$ is the cumulant associated to the truncated homogeneous growth-fragmentation  $\boldsymbol{\chi}^{(\veps)}$. 
We want to prove the following:

\begin{prop}\label{trunckappapos} For $\veps$ small enough, we have $\kappa^{(\veps)}(q)\in(0,+\infty)$ for all $q\geq 0$ and the right derivative $\dot\kappa^{(\veps)}(0^+)$ of $\kappa^{(\veps)}$
at $0$ is strictly negative.

\end{prop}

\begin{proof}
We first make a simple remark. Let any $q_0\geq 0$ such that $\kappa(q_0)<\infty.$ Then for $q\geq q_0$, we always have 
$$|\kappa^{(\veps)}(q)-\kappa(q)|\leq \int_{[-\veps,0)} (1-\e^{y})^{q_0}\Lambda(\mathrm dy),$$
 which implies that $\kappa^{(\veps)}$ converges uniformly to $\kappa$ on the interval $[q_0,\infty)$. This guarantees that its infimum also converges, and thus $\inf_{q\geq q_0} \kappa^{(\veps)}(q)>0$ for $\veps$ small enough, and this infimum is actually a minimum. Also, invertedly, if $q_0$ is such that $\kappa(q_0)=\infty$, then $\kappa^{(\veps)}$ converges uniformly to infinity on $[0,q_0].$

Assume by contradiction that, for all $\veps>0$, there exists $q_{\veps}$ such that $\kappa^{(\veps)}(q_{\veps})\leq 0$. Then, as $\veps$ tends to $0$, $q_{\veps}$ must converge to $\underline{q}=\inf\{q\geq 0: \kappa(q)<\infty\}$. Indeed, if a subsequential limit was strictly smaller or larger than $\underline{q}$, then the previous paragraph would be contradicted.

Now let $q'$ and $q''$ be such that $\underline{q}<q'<q''$. By standard convexity properties, we have
\[\frac{\kappa^{(\veps)}(q')-\kappa^{(\veps)}(q_{\veps})}{q'-q_{\veps}}\leq \dot\kappa^{(\veps)}(q'').\]
However since $\kappa^{(\veps)}(q_{\veps})\leq 0$ for all $\veps$, we also have
\[\underset{\veps\to 0}\liminf\; \frac{\kappa^{(\veps)}(q')-\kappa^{(\veps)}(q_{\veps})}{q'-q_{\veps}} \geq \underset{\veps\to 0}\liminf\;\frac{\kappa^{(\veps)}(q')}{q'-q_{\veps}}=\;\frac{\kappa(q')}{q'-\underline{q}} .\]
Recall that by monotone convergence, $\lim_{q\to \underline{q}^+}\kappa(q)=\kappa(\underline{q})>0$, so we may choose $q'$ close enough to $\underline{q}$ such that 
\[\frac{\kappa(q')}{q'-\underline{q}}>\dot \kappa(q'').\]
This  leads to a contradiction, since one readily checks that $\dot \kappa^{(\veps)}(q'')$ also converges to $\dot \kappa(q'')$ as $\veps$ tends to $0$.

We have thus proven that for $\veps>0$ sufficiently small, $\kappa^{(\veps)}(q)>0$ for all $q\geq 0$, and the fact that $\kappa^{(\veps)}(q)<\infty$ is obvious since $\Lambda_1((-\infty,-\veps))<\infty$. It remains to verify that $\dot \kappa^{(\veps)}(0^+)<0$, which is straightforward by monotone convergence when $\kappa(0)<+\infty$ and $\dot \kappa(0^+)<0$. So assume that $\kappa(0)=+\infty$. Then by monotone convergence, $\lim_{\veps\to 0} \kappa^{(\veps)}(0)= \infty$, and we may choose $\varepsilon >0$ small enough so that 
$\kappa^{(\veps)}(0)>\kappa(2)\geq \kappa^{(\veps)}(2)$. By convexity, this forces $\dot \kappa^{(\veps)}(0^+)<0$.  
\end{proof}


The construction of growth-fragmentation processes by truncation shows that ${\mathbf X}^{(\veps)}(t) \subset {\mathbf X}(t)$ for every $\veps>0$. 
An important consequence of Proposition \ref{trunckappapos}, combined with Lemma \ref{lem:truncdeath} for us is that, in order to prove Theorem \ref{thprincipal}, we can restrict ourselves to the cases where $\Lambda_1$ is finite. For this, we also need the following elementary lemma.

\begin{lemma}\label{lem:truncdeath} Let, for $\veps>0$, $\mathbf{X}^{(\veps)}$ be the truncation of $\mathbf{X}$ defined in Section \ref{subsec:construction}. Then, almost surely, $\mathbf{X}$ dies suddenly if and only if $\mathbf{X}^{(\veps)}$ dies suddenly for all $\veps>0$.
\end{lemma}

\begin{proof} Equivalently, we can just prove that, in the homogeneous case, $\boldsymbol{\chi}$  goes extinct if and only if all its truncations $\boldsymbol{\chi}^{(\veps)}$  do, so we now assume that $\alpha=0$. The direct implication is immediate, so we focus on the reverse. For $n\in\Z_+$, let $Z^{(\veps)}_n$ be the number of particles of $\boldsymbol{\chi}^{(\veps)}(n)$, and $Z_n$ be that of $\boldsymbol{\chi}(n)$. By homogeneity, these are all Galton-Watson processes (possibly taking infinite values, but that does not pose a problem, as is shown in \cite{Steph13}, Appendix B), and we also have 
\[Z_n=\underset{\veps\to 0}\lim\; Z_n^{(\veps)}.\]
We let respectively $p$ and $p_{\veps}$ be the extinction probabilities of $(Z_n,n\in\Z_+)$ and $(Z^{(\veps)}_n,n\in\Z_+)$, and $F$ and $F_{\veps}$ be their generating functions, i.e. $F(x)=\E[x^{Z_1}]$ and $F_{\veps}(x)=\E[x^{Z_1}]$ for $x>0$. We know that $p_{\veps}$ decreases to a certain limit $p'$ as $\veps$ tends to $0$, and we will show that $p'=p.$ Note that $p<1$ by Remark \ref{rem:extinction}, this implies that $\mathbb{E}[Z_1]>1$, and by monotone convergence, $\mathbb{E}[Z^{(\veps)}_1]>1$ for $\veps$ small enough, which in turn shows that $(Z^{(\veps)}_n,n\in\Z_+)$ is supercritical, and thus $p_{\veps}<1$. Again by monotone convergence, $F_{\veps}(x)$ increases to $F(x)$ for all $x\in[0,1)$, and this implies uniform convergence on the compact interval $[0,p_{\veps_0}]$, for a fixed small enough $\veps_0$. We can then take the limit in the standard fixed point relation $F_{\veps}(p_{\veps})=p_{\veps}$ to obtain $F(p')=p'.$ Thus $p'$ is a fixed point of $F$, but $p'<1$, and $F$ classically only has one fixed point apart from $1$, implying $p=p'$.
\end{proof}

The conclusion that we draw from this section is that there is no loss of generality in assuming that the measure $\Lambda_1$ is finite. This will simplify matters greatly, and we will do it from now on.

\subsection{Tilted probabilities and spinal decomposition}
We investigate in this section the additive martingale $M_q$ and what happens when we use it to bias the distribution of the homogeneous growth-fragmentation process.

\begin{prop}\label{prop:choiceq} There exist $q_+>q_->0$ such that $\dot \kappa(q_-)<0<\dot \kappa(q_+)$ and the martingales $M_{q_-}$ and $M_{q_+}$ both converge in $L^p$ for some $p>1$.
\end{prop}

\begin{proof}
We use Theorem $2.3$ from \cite{Benjamin}. This theorem states in particular that, if $\Lambda_1$ has finite total mass, then we can look at $(\boldsymbol{\chi}(n),n\in\Z^+)$ as a branching random walk, and use Biggins' theorem from \cite{Biggins92}, which gives the $L^p$-convergence under two conditions: one is $\dot \kappa(q)<\kappa(q)/q$, the other is an integral condition which is satisfied for $p>1$ small enough. Consider then $q_m>0$, the location which minimizes $\kappa$ and recall Proposition \ref{trunckappapos}.  Since $\kappa$ is strictly convex, we then have $\dot \kappa(q)>0$ for $q>q_m$, and $\dot \kappa(q)<0$ for $q<q_m$. 
Using continuity and the facts that $\kappa(q_m)>0$ and $\dot \kappa(q_m)=0$, we obtain that $\dot \kappa(q)<\kappa(q)/q$ for all $q$ close enough to $q_m$.
\end{proof}

We now consider such a $q=q_{\pm}$, and thus assume that the martingale $M_q$ converges in $L^p$ for some $p>1$.

\begin{lemma}\label{lem:poslim} The limit $M_q(\infty)$ of $M_q$ is equal to zero on the event where $\boldsymbol{\chi}$ goes extinct, and is strictly positive on the event where $\boldsymbol{\chi}$ does not become extinct.
\end{lemma}


\begin{proof} We adapt a fairly standard argument. For $n\in\Z_+$, recall that $Z_n$ be the number of particles of $\boldsymbol{\chi}(n)$ and that $(Z_n,n\in\Z_+)$ is a Galton-Watson process. Moreover, the event $\{M_q(\infty)=0\}$ is hereditary for this Galton-Watson process in the sense that $M_q(\infty)=0$ if and only if for every individual alive at time $1$, the analoguous additive martingale corresponding to the descendants of this individual also has limit zero. Its probability must then be equal to either $1$ or the probability of extinction. However the first case is excluded, since we have $\mathbb{E}[M_q(\infty)]>0$ by $L^1$-convergence. The event $\{M_q(\infty)=0\}$ then contains the event of extinction and they have the same probability, and they must almost surely be equal.
\end{proof}

We next use the additive martingale $M_q$ for $q=q_{\pm}$ to define two tilted probability measures, $\Q^+$ and $\Q^-$. Formally, denote by ${\mathcal T}$ the space of families of triples $\left((\xi_{\mathbf u}, T_{\mathbf u}, J_{\mathbf u}): {\mathbf u}\in{\mathcal U}\right)$ where $\xi_{\mathbf u}$ is a c\`adl\`ag real path (possibly with finite lifetime), $T_{\mathbf u}\in(0,\infty)$ and $J_{\mathbf u}\in(-\infty,0)$, and recall that $\pr$ is the probability measure on ${\mathcal T}$ under which the triples $(\xi_{\mathbf u}, T_{\mathbf u}, J_{\mathbf u})$ for  ${\mathbf u}\in{\mathcal U}$ are i.i.d.; more precisely each $\xi_{\mathbf u}$ is a spectrally negative L\'evy process with Laplace exponent $\Psi_2$, $T_{\mathbf u}$ has the exponential law with parameter $\Lambda_1((-\infty,0))$,  $J_{\mathbf u}$ has the law $\Lambda_1(\cdot)/\Lambda_1((-\infty,0))$, and finally $\xi_{\mathbf u}, T_{\mathbf u}, J_{\mathbf u}$ are independent. Recall also the construction of the processes $(\chi_{\mathbf u}(t): \beta_{\mathbf u}\leq t < \delta_{\mathbf u})$ from the preceding, and consider the filtration
$${\mathcal F}(t)=\sigma\left({\bf 1}_{\{\beta_{\mathbf u}\leq s < \delta_{\mathbf u}\}}\chi_{\mathbf u}(s)  : 0\leq s \leq t,   {\mathbf u}\in{\mathcal U}\right).$$
We endow the infinite binary tree with its discrete sigma-algebra (i.e. its power set) ${\mathcal P}({\mathcal U})$
and for every $t\geq 0$, we define two tilted probability measures $\Q^{\pm,t}$  on $\left({\mathcal T}\times {\mathcal U}, {\mathcal F}(t)\otimes {\mathcal P}({\mathcal U})\right)$ 
 by the following formula:
\[\Q^{\pm,t}\Big[A\times\{\mathbf{u}\}\Big]=\e^{-t\kappa(q_{\pm})}\E\Big[\chi^{q_{\pm}}_{\mathbf{u}}(t){\bf1}_A \Big], \qquad A\in{\mathcal F}(t), {\mathbf u}\in{\mathcal U}.\]
(recall that $\chi_{\mathbf{u}}(t)$ is implicitly assumed to be $0$ if we do not have $\beta_{\mathbf u}\leq t < \delta_{\mathbf u}$.) In particular, $\Q^{\pm,t}$ is the joint law of a growth-fragmentation $\boldsymbol{\chi}$ observed up to time $t$ and a randomly tagged particle ${\mathbf u}^*\in{\mathcal U}$ which is alive at time $t$, $\Q^{\pm,t}$-a.s. The following compatibility property of the laws $(\Q^{\pm,t},t\geq0)$ follows immediately from the martingale property of $M_{q_{\pm}}$ and the branching property of homogeneous growth-fragmentations. 

\begin{prop} The measures $(\Q^{\pm,t},t\geq0)$ are compatible in the sense that, for $s\leq t$, if  $\big((\boldsymbol{\chi}(r),r\in[0,t]),\mathbf{u}^*\big)$ has distribution $\Q^{\pm,t}$, then, letting $\mathbf{v}^*$ be the ancestor of $\mathbf{u}^*$ which is alive at time $s$,  $\big((\boldsymbol{\chi}(r),r\in[0,s]),\mathbf{v}^*\big)$ has distribution $\Q^{\pm,s}$.
\end{prop}

By Kolmogorov's theorem, there exists a probability measure $\Q^{\pm}$ describing the joint distribution of a growth-fragmentation $\boldsymbol{\chi}=(\boldsymbol{\chi}(t): t\geq 0)$
and a selected \emph{line of descent}  $({\mathbf u}^*(t): t\geq 0)$, that is a process with values in ${\mathcal U}$ with the property  that for every $0\leq s \leq t$, the particle ${\mathbf u}^*(s)$ is the ancestor at time $s$ of the particle ${\mathbf u}^*(t)$, such that the distribution of $\big((\boldsymbol{\chi}(r),r\in[0,t]),\mathbf{u}^*(t)\big)$ under $\Q^{\pm}$ is $\Q^{\pm,t}$.

 Under $\Q^{\pm}$, the selected line of descent $(\mathbf{u}^*(t),t\geq 0)$ serves as a spine of the process, and if we follow it we get a particular Markov process. We write $\chi^*(t)=\chi_{\mathbf{u}^*(t)}(t)$, call $\chi^*$ the \emph{selected particle},  and claim:
 
\begin{prop} \label{P24} We work under $\Q^{\pm}$ and for $t\geq 0$, we let $\xi^*(t)=\log\big(\chi^*(t)\big).$ The process $\big(\xi^*(t),t\geq 0\big)$ is a L\'evy process with Laplace exponent $\Phi_{\pm}(\cdot)=\kappa(q_{\pm}+\cdot)-\kappa(q_{\pm})$.\end{prop}

\begin{proof} For the sake of simplicity, we drop $\pm$ from the notation and simply write $q=q_{\pm}$, $\Q=\Q^{\pm}$. 

Let us show the independence and stationarity of the increments of $\xi^*$. Let $0\leq s\leq t$, we have, for appropriate functions $f$ and $G$, and using the branching property at time $s$,
\begin{align*}
&\Q\big[f(\xi^*(t)-\xi^*(s))G(\xi^*(r),r\leq s))\big] \\ 
&=\e^{-t\kappa(q)}\E\Big[\sum_{\mathbf{u}\in\mathcal{U}}\chi^q_{\mathbf{u}}(s)G\big(\log(\chi_{\mathbf{u}(r)}(r)),r\leq s\big)\sum_{\mathbf{v}\in\mathcal{U}}\left( \frac{\chi_{\mathbf{uv}}(t)}{\chi_{\mathbf{u}}(s)}\right)^qf\big(\log(\chi_{\mathbf{uv}}(t)-\log(\chi_{\mathbf{u}(s)}))\big)\Big] \\
&=\e^{-s \kappa(q)}\E\Big[\sum_{\mathbf{u}\in\mathcal{U}}\chi^q_{\mathbf{u}}(s)G\big(\log(\chi_{\mathbf{u}(r)}(r)),r\leq s\big)\Big]\e^{-(t-s)\kappa(q)}\E\Big[\sum_{\mathbf{v}\in\mathcal{U}}\chi^q_{\mathbf{v}}(t-s) f\big(\log(\chi_{\mathbf{v}}(t-s))\big)\Big]\\
&=\Q\big[f(\xi^*(t-s))\big]\Q\big[G(\xi^*(r),r\leq s)\big].
\end{align*}

We then only need to check that the moments match up with the announced Laplace exponent. For $t\geq0$ and $p\geq0$, we have
\[\Q[\e^{p\xi^*(t)}]=\e^{-t\kappa(q)}\E\Big[\sum_{\mathbf{u}\in\mathcal{U}}\chi^q_{\mathbf{u}}(t)\chi^p_{\mathbf{u}}(t)\Big]=\e^{t(\kappa(p+q)-\kappa(q))},\]
which ends the proof.
\end{proof}

In their most influential contribution, Lyons {\it et al.}  \cite{LPP} pointed at the well-known spine-decomposition of branching processes under the tilted probability measure that  is  induced by an additive martingale. Roughly speaking, it states that the descent of particles who are sibling of the spine evolve according to independent branching processes with the original (i.e. non-tilted) distribution. We shall now state a version of this spine decomposition in the setting of homogeneous growth-fragmentations. 

We work under $\Q=\Q^{\pm}$ and write ${\mathbf u}^*(\infty)=u^*_1u^*_2\cdots$ for the spine, that is the infinite word induced by the selected line of descent. For every $n\geq 1$, we denote by
$\boldsymbol{\chi}_n$ the sub-growth-fragmentation generated by the $n$-th sibling particle of the spine, namely  ${\mathbf v}^*_n=u^*_1\cdots u^*_{n-1}(1-u^*_n)$. Specifically,
 write $s=\beta_{{\mathbf v}^*_n}$ for the birth-time of that particle and  set for $t\geq 0$
$$\boldsymbol{\chi}_n(t)=\{\!\!\{\chi_{{\mathbf v}^*_n{\mathbf{u}}}(t+s):\;\mathbf{u}\in\mathcal{U}\text{ and }
\beta_{{\mathbf v}^*_n{\mathbf{u}}} \leq t+s<\delta_{{\mathbf v}^*_n{\mathbf{u}}}\}\!\!\}.
$$

\begin{lemma}\label{L2} Let $(x_n)_{n\geq 1}$ be a sequence of positive real numbers. 
Under $\Q^{\pm}$ and conditionally on $\chi_{{\mathbf v}^*_n}(\beta_{{\mathbf v}^*_n})=x_n$ for $n\geq 1$, the sub-growth-fragmentations $\boldsymbol{\chi}_n$ are independent and each has the law $\pr_{x_n}$. 
\end{lemma}

More than 20 years after the publication of  \cite{LPP}, the proof of Lemma \ref{L2} is nowadays standard and follows from calculations similar to those performed in the proof of Proposition \ref{P24}. Details are left to the interested readers.



\subsection{Starting near a boundary}

We recall that ${\mathbf X}$ denotes the growth-fragmentation process obtained from $\boldsymbol{\chi}$ by the Lamperti transformation.  
In this section, we shall mainly study effects of the Lamperti transformation under the tilted probability laws $\Q^{\pm}$.

We first consider the Lamperti transformation applied to the selected particle $\big(\chi^*(t): t\geq 0\big)$. That is, we introduce the  time-change defined by
	\[\tau^*(t) = \inf \Big\{s\geq 0:  \int_0^s \chi^*(r)^{-\alpha}\mathrm{d}r >t\Big\},\]
and the selected particle for the self-similar growth-fragmentation ${\mathbf X}$ is given by
 $X^*(t)=\chi^*(\tau^*(t))$.
We stress that $X^*$ has lifetime 
$$\zeta^*=\int_0^{\infty} \chi^{-\alpha}_{\mathbf{u}^*(r)}(r)\mathrm{d}r,$$
as $\tau^*(t)<\infty$ if and only if $t< \zeta^*$.  
Recall from Proposition \ref{P24} that under $\Q^{\pm}$, $\xi^*=\log\big(\chi^*\big)$ is a L\'evy process with Laplace exponent $\Phi_{\pm}$, and thus 
$X^*$ is a self-similar Markov process with characteristics $(\Phi_{\pm},\alpha)$.

Next observe that $\Phi_{\pm}(0)= 0$ and $\dot\Phi_{\pm}(0)=\dot \kappa (q_{\pm})$.
So under $\Q^+$ (respectively, under $\Q^-$) we see that the L\'evy process $\xi^*$ has no killing and that its expectation is positive (respectively, negative). We readily deduce the following statement from the law of large numbers.

\begin{cor} \label{C1} {\rm (i)} Suppose $\alpha>0$.
Under $\Q^+$, $\zeta^*$
is an a.s. finite random variable and 
$$\lim_{t\to \zeta^*}X^*(t)=+\infty.$$

\noindent {\rm (ii)} Suppose $\alpha<0$.
Under $\Q^-$, $\zeta^*$
is an a.s. finite random variable and 
$$\lim_{t\to \zeta^*}X^*(t)=0.$$
\end{cor}

We now arrive at a key step in the proof of local explosion. For $\alpha <0$, we consider the self-similar growth-fragmentation starting from a single particle 
with a small initial size (that is, we work under $\pr_x$ with $x\ll 1$) and show that for every $0<a<a'$, we can find a time-interval $[t,t']$ such that the probability that 
${\mathbf X}$  has particles in $(a,a')$ for all times $r\in[t,t']$ remains bounded away from $0$ as $x$ tends to $0$. A similar property holds for $\alpha >0$, except that the initial size $x$ now tends to $+\infty$. Here is the precise statement.

\begin{lemma} \label{L24} Fix $0<a<a'$. There exist  $0<t<t'$ such that:

\noindent {\rm (i)} if $\alpha<0$, then
\[\underset{x\to 0^+}\liminf \; \pr_x [{\mathbf X}(r)\cap (a,a')\neq \emptyset \hbox{ for all }t\leq r \leq t']>0.\]

\noindent {\rm (ii)} if $\alpha>0$, then
\[\underset{x\to +\infty}\liminf \; \pr_x [{\mathbf X}(r)\cap (a,a')\neq \emptyset \hbox{ for all }t\leq r \leq t']>0.\]
\end{lemma}

\begin{proof} (i) We assume $\alpha <0$ and shall first establish the assertion of the statement when $\pr _x$ is replaced by its tilted version 
 $\Q^{+}_x$. The martingale we use for this tilting transformation is 
$$M_{q_+}(t)=x^{-q_+}\e^{-t\kappa(q_+)}
\sum_{{\mathbf u}\in{\mathcal U}} \chi^{q_+}_{\mathbf u}(t).
$$
Note that, by scaling, its distribution does not depend on $x$. 

We have seen that under $\Q^+_x$, $X^*$ is a self-similar Markov process with characteristics $(\Phi_{+},\alpha)$, and started from $x$. Since $\dot \Phi_+(0^+)=\dot \kappa(q_+)>0$ and $-\alpha >0$, it is known $0^+$ is an entrance boundary for this process. That is, as its starting point $x$ tends to $0$, $X^*$ converges weakly in the sense of Skorokhod to a self-similar process
$(Y^+(t),t\geq 0)$ started from $Y^+(0)=0$ with c\`adl\`ag paths and no positive jumps,  and such that $\lim_{t\to \infty} Y^+(t)=+\infty$ a.s. 
See \cite{CabCha} or \cite{CKPR}. In particular, given any $0<c<c'$, there exist two times $0<s<s'$  such that $\pr [\forall r\in(s,s'), Y^+(r)\in (c,c')]>0$. 
Picking now arbitrary $0<a<c<c'<a'$ and $s<t<t'<s'$, the Portmanteau theorem then yields
\begin{eqnarray*}
& &\underset{x\to 0^+}\liminf \; \Q^+_x [{\mathbf X}(r)\cap (a,a')\neq \emptyset \hbox{ for all }t\leq r \leq t']\\
&\geq & \underset{x\to 0^+}\liminf \; \Q^+_x [X^*(r)\in (a,a') \hbox{ for all }t\leq r \leq t']\\
&\geq & \pr[Y^+(r)\in (c,c') \hbox{ for all }s\leq r \leq s']>0.
\end{eqnarray*}

In order to establish a similar inequality under $\pr _x$ rather than $\Q^+_x$, we use the $L^p$ convergence of the martingale $M_q$ given in Proposition \ref{prop:choiceq} for a certain $p>1$. Indeed, letting $p'=(1-1/p)^{-1}$ be the H\"older conjugate of $p$, we have
\begin{align*}
& \Q^+_x [{\mathbf X}(r)\cap (a,a')\neq \emptyset \hbox{ for all }t\leq r \leq t'] \\
&= \mathbb{E}_x [M_{q^+}(\infty){\bf 1}_{{\mathbf X}(r)\cap (a,a')\neq \emptyset \hbox{ for all }t\leq r \leq t']}] \\
&\leq \big(\mathbb{E}_x[M_{q^+}(\infty)^p]\big)^{1/p} \big(\pr_x [{\mathbf X}(r)\cap (a,a')\neq \emptyset \hbox{ for all }t\leq r \leq t']\big)^{1/p'}.
\end{align*}
Recall that $\mathbb{E}_x[M_q(\infty)^p]$ does not depend on $x$. Since we have already proved that the left-hand side is bounded from below, our claim follows.

(ii) The case $\alpha>0$ is similar with simple modifications. That is, we first establish a bound under the tilted measure $\Q^-_x$, using the fact that since $\dot \Phi_-(0^+)=\dot \kappa(q_-)<0$ and $-\alpha <0$, $+\infty$ is an entrance boundary for the self-similar Markov process with characteristics $(\Phi_{-},\alpha)$ (this follows straightforwardly from \cite{CabCha} or \cite{CKPR} by considering the inverse of the self-similar Markov process). Then we deduce the analog result under $\pr_x$, using H\"older's inequality as in (i). 
\end{proof}

We have now all the ingredients needed for the proof of Theorem \ref{thprincipal}. 
\section{Proof of Theorem \ref{thprincipal}}

Our final preparatory assertion compares the distributions of the growth-fragmentation under $\pr$ and under the tilted probability measures.

\begin{prop}\label{prop:equiv} The distribution of ${\mathbf X}$ under $\pr$ and conditionally on no sudden death is equivalent to that under $\Q^{\pm}$. 
\end{prop}
\begin{proof} Because the convergence $\lim_{t\to \infty} M_{q_{\pm}}(t)=M_{q_{\pm}}(\infty)$ holds in $L^1(\pr)$, 
it is immediately seen that the distribution of ${\mathbf X}$ under $\Q^{\pm}$ has density $M_{q_{\pm}}(\infty)$ with respect to $ \pr$. By Lemma \ref{lem:poslim}, this density is strictly positive on the event of no sudden death and is zero on the event of suddent death, and thus we have the wanted equivalence conditionally on no sudden death.
\end{proof}

In particular, in order to prove that the conclusion in Theorem \ref{thprincipal} holds a.s. under the conditional probability $\pr$ given that the self-similar growth-fragmentation ${\mathbf X}$ does not suddenly die, it is enough to show that the same assertion holds $\Q^{\pm}$-a.s.  It is more convenient for us to work under $\Q^-$ when $\alpha <0$, and under $\Q^+$ when $\alpha >0$. 

Specifically, assume first that $\alpha <0$. Recall  Lemma 
\ref{L2} and the notation there, and in particular that ${\mathbf v}^*_n$ is the $n$-th sibling particle of the selected line of descent. Since its size at birth cannot exceed the size of the selected particle $\chi^*=\e^{\xi^*}$ immediately before the jump of $\chi^*$ at which ${\mathbf v}^*_n$ is born, and since $\lim_{t\to \infty}\chi^*(t) =0$ a.s. under $\Q^-$, the size at birth of  the particle  ${\mathbf v}^*_n$ converges to $0$ as $n\to \infty$, $\Q^-$-a.s. Let us write ${\mathbf X}_n$ for the Lamperti transform of $\boldsymbol{\chi}_n$, the sub-growth-fragmentation generated by ${\mathbf v}^*_n$, and recall from  Lemma \ref{L24}, that we can find  $\veps>0$  and $x_{\veps}>0$ sufficiently small such that 
\[ \pr_x \big[{\mathbf X}(r)\cap (a,a')\neq \emptyset \hbox{ for all }t\leq r \leq t'\big]>\veps \qquad \hbox{for all }x<x_{\veps}.\]
Combining these observations with the spine decomposition under $\Q^-$ stated in Lemma \ref{L2}, and applying the Borel-Cantelli lemma, we deduce that 
 $\Q^-$-a.s. there are infinitely many integers $n$ such that ${\mathbf X}_n(r)\cap (a,a')\neq \emptyset \hbox{ for all }t\leq r \leq t'$. But the time $b_{{\mathbf v}^*_n}$ at which the particle ${\mathbf v}^*_n$ is born in the self-similar growth-fragmentation ${\mathbf X}$ converges to the lifetime $\zeta^*$ of the selected particle $X^*$,
 which is finite $\Q^-$-a.s. To complete the proof, we only consider siblings of the selected line of descent which are born at times $b_{{\mathbf v}^*_n}>\zeta^*-(t'-t)$ (a condition which holds whenever $n$ is sufficiently large), we conclude that  ${\bf X}(\zeta^*+t)$ possesses infinitely many elements in $(a,a')$, a.s. 
 
 This proves Theorem \ref{thprincipal} in the case $\alpha <0$; the proof for $\alpha >0$ is the same up to obvious modifications (in particular, one works under $\Q^+$ and the selected particle now converges to $+\infty$ at its lifetime). \qed

\section*{Acknowledgments} We would like to thank Benjamin Dadoun and Alex Watson for discussions and suggestions.

\bibliographystyle{siam}
\bibliography{frag}
\end{document}